\documentclass[11pt]{amsart}
  \usepackage{graphics} 
  \usepackage{epsfig} 
\usepackage{graphicx}  \usepackage{epstopdf}
 \usepackage[colorlinks=true]{hyperref}
 \usepackage[toc,page]{appendix}
 \usepackage[utf8x]{inputenc}
 \usepackage{amsmath, amsthm, amsfonts, amssymb}
\usepackage{epsfig, enumerate}
\usepackage{color}
\usepackage{stmaryrd}

\usepackage{hyperref}

\usepackage{times,amsfonts,amsmath,amstext,amsbsy,amssymb,amsopn,amsthm,upref,eucal}
\usepackage[T1]{fontenc}

\newcommand \Ker {{\rm Ker}}
\newcommand \R {{ \mathbb R}}
\newcommand \RR {{ \mathbb R}}

\newcommand {\abs}[1]{\lvert#1\rvert}
\newcommand \C {{ \mathbb C}}
\newcommand \N {{ \mathbb N}}
\newcommand \T {{ \mathbb T}}
\newcommand \E {{ \mathcal E}}
\newcommand \cL {{ \mathcal L}}
\newcommand \g {{ \mathfrak g}}

\renewcommand{\sl}{\operatorname{\mathfrak s\mathfrak l}}

\newcommand \GN {{ \Gamma\setminus N}}

\newcommand \ci{{C^\infty}}

\def \al{ \alpha}
\def \la{\lambda}

\newtheorem {lemma} {Lemma}[section]
\newtheorem {thm} {Theorem}
\newtheorem {prop} {Proposition}[section]
\newtheorem{cor}{Corollary}[section]
\newtheorem{rem}{Remark}

\newtheorem{definition}{Definition}
\newtheorem{conj}{Conjecture}

\title[Globally hypoelliptic abelian actions] 
      {On globally hypoelliptic abelian actions and their existence on homogeneous  spaces}

\author[Danijela Damjanovic, James Tanis and Zhenqi Wang]{Danijela Damjanovic, James Tanis and Zhenqi Wang}

 \email{ddam@kth.se}
 \email{jhtanis@mitre.org}
 \email{wangzq@math.msu.edu}

\thanks{The first author is supported by Swedish Research Council  grant   VR-2015-04644. The third author is supported by NSF grant DMS-1346876. Approved for Public Release; Distribution Unlimited. Public Release Case Number 19-2033.\\}

\begin{document}
\maketitle


\bigskip


\begin{abstract}
We define globally hypoelliptic smooth $\mathbb R^k$ actions as actions whose leafwise Laplacian along the orbit foliation is a globally hypoelliptic differential operator. When $k=1$, strong global rigidity is conjectured for such actions by Greenfield-Wallach and Katok: every such action is smoothly conjugate to a Diophantine flow on the torus. The conjecture has been confirmed for all homogeneous flows on homogeneous spaces \cite{FFRH}. In this paper we conjecture that among homogeneous $\mathbb R^k$ actions ($k\ge 2$) on homogeneous spaces globally hypoelliptic actions exist only on nilmanifolds. We obtain a partial result towards this conjecture: we show non-existence of globally hypoelliptic $\mathbb R^2$ actions on homogeneous spaces $G/\Gamma$, with at least one quasi-unipotent generator, where $G= SL(n, \mathbb R)$. We also  show that the same type of actions on solvmanifolds are smoothly conjugate to homogeneous actions on nilmanifolds. 
\end{abstract}

\section{Introduction and result}

\subsection{Globally hypoelliptic $\mathbb R^k$ actions}

Let $\rho$ be a smooth  $\R^k$ action by diffeomorphisms of a smooth manifold  $M$ generated by $k$ smooth commuting vector fields $X_1, \dots X_k$.
We use the same notation $X_i$ to denote the Lie derivative  $\cL_{X_i}$ associated to the vector field $X_i$, $i=1, \dots, k$.  The second order differential operator    $$L_\rho=\sum_{i=1}^kX_i^2$$
is called \emph{the leafwise Laplacian associated to the action $\rho$}. It is the Laplacian operator associated to the orbit foliation of the action $\rho$.

For a differential operator   $L$ on a compact connected smooth manifold $M$ we define now the property of \emph{global hypoellipticity}.  Denote by $\E(M)$ the space $\ci(M)$ of smooth functions and by $\E'(M)$ the space of distributions.
\begin{definition} Operator $L$ is globally hypoelliptic (GH) if every distributional solution $u$  to the equation $Lu=f$  is in $\E(M)$, if $f\in \E(M)$.
\end{definition}

We say that an $\R^k$ action $\rho$ is \emph{globally hypoelliptic} (GH) if its associated leafwise Laplacian $L_\rho$ is globally hypoelliptic.

Examples of GH flows are Diophantine flows on tori. On the  torus $\T^n$ such a flow is given by constant coefficient vector field: $X_\alpha= \sum_{i=1}^{n} \al_i\frac{\partial}{\partial x_i}$,  where vector $\alpha=(\alpha_1, \dots, \alpha_n)$ is a Diophantine vector in $\R^n$. Examples of GH $\R^k$ actions are again actions generated by (simultaneously)  Diophantine vector fields on tori.  For $k= 2$ however, there exist examples of GH actions on 2-step nilmanifolds \cite{CR1, D} (the known examples are listed in Section \ref{examples}). It is expected that  for $k> 2$ there are examples on (higher step) nilmanifolds as well, although such examples are not yet known.  All these examples are \emph{homogeneous actions} (i.e. actions by multiplication by group elements) on homogeneous space of the form $G/\Gamma$ where $G$ is a Lie group and $\Gamma$ is a co-compact lattice in $G$.  It was found already by Greenfield and Wallach in the 70's that GH property for a smooth flow on a compact manifold is a very restrictive property. In particular, it implies that the flow is conservative (i.e. it preserves a smooth nowhere vanishing volume form) and moreover, the space of the flow invariant distributions is 1-dimensional \cite{GW2}. 
This gives strong dynamical restrictions on GH flows:  any GH flow is uniquely ergodic and minimal.  These conclusions from \cite{GW2} generalise in a straightforward way to general $\R^k$ GH actions (see Proposition \ref{trivial_kernel}). In particular,  every GH action  of  $\R^k$ is conservative, uniquely ergodic and minimal.  Already these properties restrict significantly the dynamics of GH actions.  In the case of $\R$ actions, the GH property is  actually extremely strong in the sense that  conjecturally it implies very strong global rigidity.  Namely, Greenfield and Wallach in \cite{GW2} as well as Katok in \cite{K} conjectured: \color{black}

\begin{conj}\label{GWK}[Greenfield-Wallach, A. Katok]
Any  globally hypoelliptic  $\R$ action on a smooth connected compact manifold $M$ is smoothly conjugated to a Diophantine flow on the torus.
\end{conj}

This conjecture is proved in several special situations:  when $M$ is the torus (see \cite{Kocsard} and references therein); when $\rm dim(M)\le 3$ \cite{F}, \cite{Kocsard}; and when $M$ is a homogeneous finite volume space $G/D$ (where $G$ is a connected Lie group and $D$ a closed subgroup of $G$ such that $G/D$ has a finite $G$-invariant smooth measure)  and $\rho$ is a homogeneous flow \cite{FFRH}. Other advances in the direction of Conjecture \ref{GWK}  on general manifolds are \cite{RHRH}, \cite{FP}.

In this paper we raise a question concerning existence of globally hypoelliptic  $\R^k$ actions when $k\ge 2$. Even though in this case the orbit foliation of the action has dimension greater than 1, and the  minimality of orbit foliation should significantly restrict the class of manifolds where such actions can exist, still the implications of GH property to cohomology over the GH group actions are significantly stronger in the case of flows than in the case of $\R^k$ actions where $k\ge 2$ (see Section \ref{coh-discussion}).  This is the main reason that in the case of $\R^k$ actions when $k\ge 2$ we state the following conjecture only in the homogeneous set-up:

\begin{conj}\label{conj_HR} Let $\rho$ be a homogeneous  $\R^k$ action on a finite volume space $G/D$, where $G$ is a connected Lie group and $D$ a closed subgroup in $G$. If $\rho$ is globally hypoelliptic then it is smoothly conjugate to a homogeneous action on
an $r$-step nilmanifold where $k\ge r$.
\end{conj}

 In Section \ref{nilmanifold}
we comment in more detail the relation $k\ge r$ between the rank of the action and the step of the nilmanifold in the conjecture above.

\begin{rem}\label{WCF}We note here that there has been another notion of global hypoellipticity which has been considered in the literature for a collection of smooth vector fields (see for example \cite{CR, CR1}). Namely, a system of  vector fields $\{X_1,  \dots, X_k\}$ is said to be globally hypoelliptic (GH) if whenever the system of equations $X_1u=f_1, \dots, X_ku=f_k$ has a distributional solution $u$ for $C^\infty$ functions $f_1, \dots, f_k$, then $u$ is also $C^\infty$.

It can be observed directly that if $L_\rho$ is GH, then the system of vector fields is also GH: existence of a distributional solution $u$ for $X_1u=f_1, \dots, X_ku=f_k$  implies that $u$ solves the equation $L_\rho u= \sum_{i=1}^kX_if_i$, where the right hand side is a smooth function, so by GH of $L_\rho$  we conclude $u$ is smooth.

The converse however does not hold, namely if the system of vector fields $\{X_1,  \dots, X_k\}$ is GH, the leafwise Laplacian of the action $\rho$ need not be GH. One example of this situation is action by left multiplication by the abelian group of diagonal matrices in $SL(3, \mathbb R)$ on the manifold $SL(3, \mathbb R)/ \Gamma$ where $\Gamma$ is a uniform lattice in $SL(3, \mathbb R)$. This action is called the Weyl chamber flow, and it follows from the rigidity result in  \cite{KS0} that its generating system of vector fields is GH. However, this action is not minimal (closed non-trivial  invariant sets are discussed in \cite{LW} for example) and therefore its leafwise laplacian is not GH, so the action is not GH according to our definition.

For $\R$ actions, that is for flows, the notions of GH of the leafwise Laplacian and GH of the generating vector field are clearly equivalent. However, for $\R^k$ actions, as the example above shows, there is quite a distinction between these two notions.
We do not expect  any global rigidity to hold for $\R^k$ actions which satisfy this significantly  weaker condition, that the system of generating vector fields  is GH.
\end{rem}

\color{black}

\subsection{Main results}

In this paper we prove Conjecture \ref{conj_HR} in case of certain class of homogeneous spaces and certain type of actions.

For any generator $X$ of the given $\mathbb R^k$-action $\rho$ on $G/D$ let $ad(X)$ be the induced operator on the Lie algebra $\frak g$ of $G$. Let $\frak g^\lambda$ denote the generalized eigenspaces for the action of $ad(X)$ on $\frak g$ . Then $\frak g$ is the direct sum of $\frak g^\lambda$s. We call the generator $X$ \emph{quasi-unipotent} if all $\lambda$ lie on the unit circle and we call it \emph{partially hyperbolic} otherwise.

\begin{thm}\label{main}
Let $G/D$ be a finite volume space, where $G= SL(n, \R)$ and $D$ is a closed subgroup in $G$. Let $\rho$ be  an $\R^2$ action of a subgroup $H<G$ by left multiplication such that $\rho$  has a quasi-unipotent generator. Then $\rho$ is not globally hypoelliptic.
 \end{thm}

 The following is essentially a  corollary of results in \cite{FFRH}.

\begin{thm}\label{main2}
Let $R/D$ be a finite volume space, where $R$ is a solvable Lie group and $D$ a closed subgroup in $R$.  Let  $\rho$ be  an $\R^k$ action of a subgroup $H<R$ by left multiplication such that $\rho$  has an ergodic quasi-unipotent generator. Then $\rho$ is smoothly conjugate to a homogeneous action on a nilmanifold.
 \end{thm}

 \begin{rem}
 In Section \ref{reduction} we describe how proving the statement in Conjecture \ref{conj_HR}  reduces to proving a generalisation of Theorem \ref{main} to the case when $G$ is any semisimple Lie group and $\rho$ any $\R^k$ homogeneous action, and a generalisation of Theorem \ref{main2} to the case of any $\R^k$ action.
  \end{rem}

\subsection{Structure of the paper}

In Section \ref{prel} we discuss properties of GH actions, and existing examples on nilmanifolds. We also discuss reduction of the problem in Conjecture \ref{conj_HR} to the cases which contain our main results. In Section \ref{thm_main} we prove Theorem \ref{main} and in Section \ref{thm_main2} we prove Theorem \ref{main2}.

\section{Preliminaries}\label{prel}

\subsection{Properties of GH actions}\label{properties}
 Let $\Ker L=\{D\in \mathcal E'(M): D(Lf)=0 \,\, \mbox{for all}\,\, f\in \E(M)\}$ denote the space of invariant distributions of a differential operator $L$.

The following proposition will be crucial for the results of this paper. It shows that global hypoellipticity of the leafwise Laplacian $L_\rho$ for a smooth $\R^k$ action not only implies preservation of nowhere vanishing smooth volume form,  but also very strong \emph{distributional unique ergodicity} i.e. the property that the only action-invariant distribution is the smooth invariant volume form.

\begin{prop} \label{trivial_kernel}  Let $\rho$ be a smooth $\R^k$ action. If $L_\rho$ is GH, then  $\rho$ preserves a nowhere vanishing smooth volume form \color{black} and $\dim(\cap_{i=1}^k \Ker X_i))=1$.
\end{prop}

\begin{proof}
The main goal here is to show existence of a $\rho$ invariant nowhere vanishing smooth volume form. Once this is known, the fact that $\dim(\cap_{i=1}^k \Ker X_i))=1$ has been proved in \cite{D}, following essentially the similar proof in the flow case in \cite{GW2}. When $\rho$ is a GH flow, existence of smooth invariant nowhere vanishing volume form  is proved in \cite{GW2}. Here we apply the same approach as in  \cite{GW2} due to the fact that $\rho$ is GH. Namely, in \cite{GW1} it is shown that if a differential operator $L$ is GH then its distributional kernel $\rm Ker L$ is finite dimensional. Thus, since $\rho$ is GH, we conclude that $\Ker L_{\rho}$ is finite dimensional. It is clear that $\cap_{i=1}^k \Ker X_i$ is contained in $\Ker L_{\rho}$, therefore $\cap_{i=1}^k \Ker X_i$ is finite dimensional, and it is spanned by finitely many smooth volume forms $\mu_1, \dots, \mu_l$, since due to GH property all elements in $\Ker L_{\rho}$ are smooth. The set $\cap_{i=1}^k \Ker X_i$ is not empty because there always exists a $\rho$ invariant probability measure by the Markov-Kakutani theorem for abelian actions \cite[Section 4]{Zimmer}. By GH, this measure has to be smooth. If there is a subset $S$ in $M$ where all $\mu_1, \dots, \mu_l$ vanish,  then this subset is clearly compact and $\rho$-invariant. Therefore, again by the Markov-Kakutani theorem, there exists an invariant probability measure $\mu_S$, which extends trivially to a $\rho$-invariant probability measure on $M$, and therefore has to be  a smooth measure in   $\cap_{i=1}^k \Ker X_i$. But this measure is a probability, thus it cannot lie in the span of the set   $\mu_1, \dots, \mu_l$, which gives a contradiction. Therefore, there is a smooth measure in the span of $\mu_1, \dots, \mu_l$ which is nowhere vanishing, and this is the nowhere vanishing $\rho$-invariant volume form we were looking for.


\end{proof}
 \color{black}


Direct consequence of Proposition \ref{trivial_kernel} is:
\begin{cor}\label{unique_ergodicity}
If $L_\rho$ is GH, then $\rho$ is uniquely ergodic and minimal $\mathbb R^k$ action.
\end{cor}

\begin{rem}\label{smooth}
In particular, Proposition \ref{trivial_kernel} implies that every action invariant distribution must be smooth. We will use this fact in the upcoming proofs.
\end{rem}

\begin{rem} Group actions whose distributional kernel contains only an ergodic invariant probability measure are labelled {distributionally uniquely ergodic (DUE)} in \cite{AFK}, where diffeomorphisms with this property were obtained on manifolds different than tori. This indicates that DUE condition is significantly weaker than GH.
\end{rem}

\subsection{Examples of $\mathbb R^k$ globally hypoelliptic actions}\label{examples}
We describe here two existing classes of globally hypoelliptic $\mathbb R^k$ actions. As far as we know, there are no other examples in the literature.

\vspace{0.2in}

\noindent {\bf Torus examples.} First class of examples are on tori $\mathbb T^n$. Let $X_{\al^1}, \dots, X_{\al^k}$ be $k$ constant coefficient vector fields on $\mathbb T^n$.  If $x_i, \, i=1, \dots, n$ are coordinates on $\mathbb T^n$, then each $X_{\al^j}= \sum _{l=1}^{n} \al_l^j \frac{\partial}{\partial x_l}$, where for each $j=1, \dots, k$, $\al^k=(\al_1^j, \dots, \al_n^j ) $ are vectors in $\mathbb R^n$.

Vectors $\al^1, \dots, \al^k$ are simultaneously Diophantine if there exist constants $C$ and $\tau$ such that for all non-zero $m\in \R^n$,
$$\max\{|\al^1\cdot m|, \dots |\al^k\cdot m|\}\ge C\|m\|^{-\tau}$$
If $k=1$ and the above inequality holds, then the vector $\al^1$ is Diophantine.

It is an exercise in Fourier analysis to show that If $\al^1, \dots, \al^k$ are simultaneously Diophantine, then the $\R^k$ action $\rho$ generated by $X_{\al^1}, \dots, X_{\al^k}$ is globally hypoelliptic.

\vspace{0.2in}
\noindent {\bf Nilmanifold examples.} The examples we describe here were found to be globally hypoelliptic in \cite{CR1}, and in \cite{D} global hypoellipticity was used for obtaining a local rigidity result.

Let $\frak N$ be a 2-step rational nilpotent Lie algebra and let $N$ be the corresponding connected simply connected Lie group. Let $\Gamma$  be a (cocompact) discrete subgroup of $N$ and let $M = \Gamma\setminus N$.
We can choose  a linear basis $Y_1 , \dots,Y_q$ and $Z_1, \dots, Z_p$  for $\frak N$ (selected from $\log \Gamma$) so that:
\begin{itemize}
\item $Y_1 +[N,N],\dots,Y_q +[N,N]$ is a basis for $N/[N,N]$, and

\item  $Z_1,\dots,Z_p$ is a basis for $[N,N]$.

\end{itemize}
We define $\R^2$ action $\rho$ as follows: for vectors
$a = (a_1,\dots,a_q) \in  \R^q $ and $b = (b_1,\dots,  b_p) \in  \R^p$, action $\rho$ is generated by the commuting pair of constant coefficient vector fields
$X_1 :=a_1Y_1+\dots+a_qY_q$ and $X_2 :=b_1Z_1+\dots+b_pZ_p$.
If both vectors $a$ and $b$ are Diophantine, then it was proved in \cite{CR}, that the action $\rho$ is globally hypoelliptic.



\subsection{Comments on existence of GH actions on nilmanifolds}\label{nilmanifold}

 Let $N$ be nilpotent Lie group of step $r$ and let $\mathfrak N$ be its Lie algebra. Let $\mathfrak N_j=[\mathfrak N, \mathfrak N_{j-1}]$ , $j=1, \dots, r$ denote the lower central series of $\mathfrak N$.  Let $\GN$ be a compact nilmanifold.

Let $\rho$ denote the action generated by $X_1, \dots, X_k$ and let $L_\rho$ denote the leafwise Laplacian $\sum_{i=1}^k X_i^2$. As noted in Remark \ref{WCF}, if $L_\rho$  is GH then the system $\{X_1, \dots, X_k\}$ is GH.

In \cite{CR} and \cite{CR1} Cygan and Richardson conjecture the following:
\begin{conj} \label{conjCR} The system  $\{X_1,  \dots, X_k\}$ of constant coefficient vector fields on $\GN$ \color{black} is GH if and only if the following two conditions hold:

i) The system is GH on the associated torus.

ii) Let $\mathfrak  L$ denote the Lie subalgebra spanned by $X_1, \dots, X_k$. For each non-zero integral functional $\la\in (\mathfrak N_j/\mathfrak N_{j+1})^*$,
\begin{equation}\label{higher-step}
\la(\mathfrak L\cap \mathfrak N_j+\mathfrak N_{j+1})\ne 0, \, \, \, \, j=1, \dots, r-1
\end{equation}
($\la\in \mathfrak N_j^*$ is integral if $\la(\log \Gamma\cap \mathfrak N_j)\subset \mathbb Z$. )
\end{conj}

This conjecture is proved in \cite{CR} under additional condition that for every infinite dimensional representation of $\GN$,
the corresponding coadjoint orbit is either flat  or the representation \color{black} is inducable from a polarization of codimension one in $\mathfrak N$. In particular, the conjecture holds for 2-step nilmanifolds, since all orbits are flat, and also holds for any nilmanifolds of higher step which have all orbits flat.

Whenever the Conjecture \ref{conjCR} holds, it follows that if an $\mathbb R^k$ homogeneous action on $\GN$ has GH leafwise Laplacian then the rank of the action $k$ must be at least the number of steps $r$ of $N$. Otherwise from the condition $ii)$ above, the system $\{X_1, \dots, X_k\}$  is not GH, thus the corresponding leafwise Laplacian is not GH.

In particular, in the setting of Theorem \ref{main} this implies that the action $\rho$ with a globally hypoelliptic $L_\rho$ is up to smooth conjugacy an action by translations on a nilmanifold which is either a 2-step nilmanifold, or a nilmanifold which has an infinite dimensional representation which is neither flat nor is inducible from a polarisation of codimension one in $\frak N$. We believe that the latter case can be eliminated and that $\mathbb R^2$ actions with GH $L_\rho$  up to smooth conjugacy exist only on 2-step nilmanifolds, in which case they are smoothly conjugate to Diophantine examples described in Section \ref{examples}.

\subsection{Cohomological properties of GH actions}\label{coh-discussion}

The first cohomology $H^1_\rho(\ci(M))$ over a smooth $\R^k$ action with coefficients in smooth functions $\ci(M)$ is defined as usual: it is the quotient of smooth cocycles factored by smooth coboundaries, for more details we refer to \cite{D}. It is an important feature of a dynamical system for many reasons; in particular, having a finite dimensional $H^1_\rho(\ci(M))$ can lead to classification of skew product extensions, or even classification of perturbations of $\rho$. In the case of flows there is only one known situation where $H^1_\rho(\ci(M))$ is finite dimensional and this is the case of Diophantine flows on the torus. In this case $\dim H^1_\rho(\ci(M))=1$ and this fact is instrumental in the KAM type local classification results for Diophantine flows on the torus. It was conjectured by Katok \cite{K} that if $\dim H^1_\rho(\ci(M))=1$ for a smooth flow $\rho$, then $\rho$ is smoothly conjugate to a Diophantine flow on the torus. Much later in \cite{CC} it was proved that this conjecture is actually equivalent to Conjecture \ref{GWK}, i.e. that condition $\dim H^1_\rho(\ci(M))=1$ for flows is equivalent to global hypoellipticity. This fact plays an important role in the works towards  Conjecture \ref{GWK}.

In the case $\rho$ is an $\R^k$ action with $k\ge 2$, the corresponding cohomological condition would be $\dim H^1_\rho(\ci(M))=k$ and it is not true that it is equivalent to global hypoellipticity. Namely, the Weyl chamber flow examples (these are homogeneous $\R^k$ actions, $k\ge 2$, on spaces $SL(k+1, R)/\Gamma$, where $\Gamma$ is a cocompact lattice) enjoy the property $\dim H^1_\rho(\ci(M))=k$ \cite{KS0}, while they are not globally hypoelliptic because their orbit foliation is not minimal \cite[Theorem 2.13]{PR}. It is not known, but presumably it should be true, that global hypoellipticity implies  $\dim H^1_\rho(\ci(M))=k$. The following results on cohomology over GH actions was obtained in \cite{D}:

\begin{prop}[\cite{D}] \label{1coh-trivial} Let $\rho$ be a smooth volume preserving globally hypoelliptic $\mathbb R^k$ action.
 If  $\rm Im (L_\rho)$ is a closed space, then $\dim H^1_\rho(\ci(M))< \infty$.  If in addition $\Ker L_\rho$ is 1-dimensional spanned by the invariant volume form, then $\dim  H^1_\rho(\ci(M))=k$ i.e. every cocycle with trivial average is a coboundary.
\end{prop}

It may be possible that globally hypoelliptic $\R^k$ actions with $k\ge 2$ exist in a non-homogeneous setting.

\subsection{Ergodic generators of $\rho$ and discreteness of $D$}\label{discreteD}
 By Corollary \ref{unique_ergodicity} any $\mathbb R^k$-action $\rho$ with globally hypoelliptic $L_\rho$ is uniquely ergodic, and therefore ergodic. We recall now a very general result for ergodic $\mathbb R^k$ actions proved by Pugh and Shub in \cite{PS} concerning existence of ergodic elements in ergodic $\mathbb R^k$-actions:

 \begin{thm}[Theorem 1 \cite{PS}]
 If $\mathbb R^k$ acts ergodically on the measure space $(M, \mu)$, and $L^2(M, \mu)$ is separable, then all the elements of $\mathbb R^k$, off a countable family of hyperplanes, are ergodic.
 \end{thm}
 Since the space $G/D$ caries a finite measure, separability of the corresponding $L^2$ space is clear. So we have the following useful fact:
 \begin{cor} \label{ergodic_generators} For any  $\mathbb R^2$ action $\rho$ on $G/D$ we can make a choice of two ergodic 1-parameter subgroups which generate $\rho$.
 \end{cor}

When $G$ is semi-simple,  the following  proposition proved  in \cite{FFRH} allowes reduction to the case when $G$ has a finite center and $D$ is discrete:
 \begin{prop}[Proposition 3.11 \cite{FFRH}]\label{Disdiscrete}
 If $G$ is a connected semi-simple group and $G/D$ a finite volume space, then existence of an ergodic homogeneous flow on $G/D$ implies that the connected component of the identity of $D$ in $G$ is normal in $G$.  This implies that we may assume that $G$ has a finite center and that $D$ is discrete.
 \end{prop}

In conclusion, when proving our main results we may assume that $D$ is a lattice in $G$ and that $\rho$ has ergodic generators.

\subsection{Reduction of statement of Conjecture \ref{main2} to the the semisimple and solvable case}\label{reduction}
In this section we explain that obtaining Conjecture \ref{main2} reduces to considering the two cases: when $G$ is semisimple and when $G$ is solvable. This reduction is for the most part the same as in \cite{FFRH}, and we describe it here for completeness.

Let $G$ be a connected, simply connected Lie group.
Let $\mathfrak g=\rm Lie \, G$ denote the Lie algebra of $G$. Every homogeneous  $\mathbb R^k$ action $\rho$ is given by a Lie algebra homomorphism from the Lie algebra $\mathfrak h$ of the acting group (in our situation this is $\mathfrak h=\rm Lie \, \mathbb R^k$), to  $\g$. In general not every Lie algebra homomorphism $\mathfrak h\to \mathfrak g$  extends to a Lie group homomorphism $H\to G$,  but when  $H$ is connected and simply connected, as in this case, then it is true that every such Lie algebra homomorphism extends and defines an $H=\mathbb R^k$ action.

Consider the Levi decomposition $G=L\ltimes R$ of the Lie group $G$, where $L$ is semisimple and $R$ is the solvable radical of $G$. Let $G_\infty$ be the smallest connected normal subgroup of $G$ containing the Levi subgroup $L$.


If $G/D$ admits an ergodic one-parameter flow, then the product ${RD}$ is closed \cite[Lemma 2.2.8]{KSS}. The maximal semisimple quotient space of $G/D$ is $ G/RD$ and the maximal solvable quotient is $G/\overline{G_\infty D}$.


It is well known that a homogeneous flow on $G/D$ is ergodic iff so are the projected flows on the maximal semisimple quotient  $ G/RD$, and on the maximal solvable quotient $G/\overline{G_\infty D}$ (cf. \cite[Theorem 2.2.9]{KSS}).






The strategy of the proof of Conjecture \ref{conj_HR} is to consider the projected action to the semisimple and solvable factors. In order to do this, one can make use of the following proposition.
\begin{prop}\label{reductionprop}
Let $\rho$ be an $\mathbb R^k$ homogeneous action on $G/D$.Suppose that $\rho$ projects smoothly onto an $\mathbb R^k$ action $p(\rho)$ on $G_1/D_1$ via an epimorphism $p:G\to G_1$ such that $D\subset p^{-1}(D_1)$. Then there is bounded, linear and injective map $P: \E'(G_1/D_1)\to \E'(G/D)$  such that:

a)  $P: Ker L_{p(\rho)}\to Ker L_{\rho}$

b) $P: \cap_{i} Ker \, p_*(X_i)\to \cap_i  Ker \, p_*(X_i)$

c) If $\rho$ is distributionally uniquely ergodic, then $p(\rho)$ is distributionally uniquely ergodic.


\end{prop}
Parts a) and b) are  proved in \cite{FFRH} and part c) follows easily from a) and b).

Therefore, to prove Conjecture \ref{conj_HR} one can first consider the maximal semisimple factor. If Theorem \ref{main} is generalised to the case of a general semisimple $G$, it means that if $\rho$ is a globally hypoelliptic action on a finite volume homogeneous space, then the maximal semisimple factor is trivial. This would reduce $\rho$ to the solvable factor. If Theorem  \ref{main2} is generalised to any homogeneous action $\rho$, then the conclusion of Conjecture \ref{conj_HR} follows.







\section{Proof of Theorem \ref{main}}\label{thm_main}

As explained in Section \ref{discreteD} we may assume that $D$ is a lattice in $G$ and that the action $\rho$ has ergodic generators.

In this section we assume $G= SL(n, \mathbb R)$ and $\rho$ is an $\mathbb R^2$ action on $M=G/D$.  
Then $\rho$ is generated by two elements of the Lie group of $G$, which we call \emph{generators of $\rho$.}  The  following two cases are proved separately:

-Both generators of $\rho$ are quasi-unipotent.

-Action $\rho$ has a partially hyperbolic and a quasi-unipotent generator.

Our strategy is to prove that the given homogeneous action $\rho$ has non-smooth $\rho$-invariant distributions. By applying the Proposition \ref{trivial_kernel}, this implies that $\rho$ is not globally hypoelliptic.  We achieve this by looking into irreducible, unitary representations of the subgroups to which $\rho$ embeds.


For this reason, we will repeatedly use a theorem of Kolmogorov-Mautner regarding the decomposition of unitary representations, \cite{S}.
Specifically, let $\mathcal H$ be a unitary representation of a Lie group $L$ with Lie algebra $\mathfrak{h}$.
Then $\mathcal H$ 
decomposes into a direct integral of unitary representations $\mathcal H_{\sigma}$
\begin{equation}\label{eq:decompose_H}
\mathcal H = \int_{Z} \mathcal H_{\sigma} d\mu(\sigma),
\end{equation}
where $\mathcal H_{\sigma}$ is a direct sum of an at most countable number of irreducible, unitary representations with parameter $\sigma$, $Z \subset \RR^2$ and $(Z, d\mu)$ is a measure space.  

Then let $W := (W_i)_i$ be an orthogonal subset of $\mathfrak{h}$ and define 
\begin{equation}\label{eq:laplace}
\triangle := \triangle(W) = -\sum_i W_i^2\,.
\end{equation}  
The operator $\triangle$ is essentially self-adjoint and satisfies $\triangle \geq 0$.
So by the spectral theorem, the operator $(I+\triangle)^{s/2}$ is defined
for all $s > 0$.
We take the Sobolev space of order $s > 0$ of $\mathcal H$ to be the Hilbert subspace $\mathcal{H}^s$ of $\mathcal{H}$ to be the maximal domain of the operator $(I+\triangle)^{s/2}$ equipped with the inner product
\[
 \langle f, g\rangle_{s} := \langle (1 + \triangle)^s f, g\rangle_{\mathcal{K}}.
\]
 We furthermore set $\mathcal H^\infty := \cap_{s \geq 0} \mathcal H^s$,
 and let $\mathcal H^{-\infty}$ be its distributional dual space.

All operators in the enveloping algebra are decomposable with respect to the
direct integral decomposition \eqref{eq:decompose_H}, so for $s \geq 0$,
\begin{equation}\label{equa:SL2R)^2-unitary_decomp}
\mathcal H^s = \int_{Z} \mathcal H_{\sigma}^s d\mu(\sigma)\,,
\end{equation}
and
$
\mathcal H_{\sigma}^s:=\mathcal H^s\cap\mathcal{H}_\sigma
$
is endowed with the inner product induced
from $\mathcal H^s$.

The above definition for Sobolev Hilbert spaces of Hilbert Sobolev spaces follows Section~2.2 of \cite{FF1}, for example.

\subsection{Action $\rho$ has two quasi-unipotent generators}

\begin{lemma}
Let $n \geq 4$, and let $\rho$ be an
$\R^2$ action by unipotents
on $SL(n, \R)$.  Then $\rho$
embeds in $(SL(2, R) \ltimes \R^2)\ltimes\RR^3$ or
$SL(2, \R) \times SL(2, \R)$.
\end{lemma}
\begin{proof}
Let $\Delta$ be the root system for $\sl(n, \R)$.
Let $\{u_\gamma\}_{\gamma\in \Delta}$ be elements of the root spaces
associated to $\Delta$ such that for any $\gamma_1, \gamma_2 \in \Delta$,
$$
[u_{\gamma_1}, u_{\gamma_2}]  =
\left\{ \begin{aligned}
& 0 &  \text{ if } \gamma_1 + \gamma_2 \notin \Delta\,; \\
& u_{\gamma_1+ \gamma_2} & \text{ otherwise}\,.
\end{aligned}
\right.
$$
Now by assumption on $\rho$, there are commuting nilpotent elements
 $u_\alpha$ and $u_\beta$ in  $\{u_\gamma\}_{\gamma\in \Delta}$
 such that for all $(t, s) \in \R^2$
\[
\rho(t, s) =  \exp(t u_\alpha)\exp(s u_\beta)\,.
\]

Let $X = [u_\alpha, u_{-\alpha}]$.  Because $[u_\alpha, u_\beta] = 0$,
 the Jacobi identity and gives
\begin{equation}\label{equa:jacobi_ident}
[u_\beta, X] = [u_\alpha, [u_\beta, u_{-\alpha}]]\,.
\end{equation}
There are two cases.
If $\beta - \alpha \in \Delta$,
then
\[
\eqref{equa:jacobi_ident}  = - [ u_\alpha, u_{\beta-\alpha}] = - u_\beta\,.
\]
We note that $(\beta - \alpha)+\beta\notin \Delta$, which gives 
\begin{align*}
 [ u_\beta, u_{\beta-\alpha}]=0. 
\end{align*}
Further, by using the Jacobi identity again, we have 
\begin{align*}
 [u_{\beta-\alpha}, X]=[u_{\beta-\alpha}, [u_\alpha, u_{-\alpha}]]=u_{\beta-\alpha}.
\end{align*}
Then all the commutators relations listed above show that $\rho$ embeds into $SL(2, \R) \ltimes \R^2$.
\color{black}

As $n \geq 4$,
there is an $\R^3$ action on $SL(2, \R) \ltimes \R^2$
such that $SL(2, \R) \ltimes \R^2$ is a subgroup of $(SL(2, \R) \ltimes \R^2) \ltimes \R^3$ in $SL(n, \R)$.
So $\rho$ embeds into $(SL(2, \R) \ltimes \R^2) \ltimes \R^3$.

If $\beta - \alpha \notin \Delta$, then
$[u_\beta, u_{-\alpha}] = 0$, so
\[
[u_\beta, X] = 0\,.
\]
And because $[u_\alpha, u_\beta] = 0$,
we have $-\alpha - \beta \notin \Delta$,
so the Jacobi identity gives
\[
[u_{-\beta}, [u_{\alpha}, u_{-\alpha}]] = - [u_{-\alpha}, [u_{-\beta}, u_\alpha]]\,.
\]
Then we get $-\beta + \alpha \notin \Delta$.
Hence,
\[
[u_{-\beta},X] = 0\,.
\]
It follows that $\rho$ embeds into a copy of $SL(2, \R) \times SL(2, \R)$.
\end{proof}

\subsubsection{$(SL(2, \mathbb R)\ltimes \mathbb R^2)\ltimes\RR^3$}\label{sect:semidirect}

Let $\mathcal H$ be a unitary representation of $(SL(2, \R) \ltimes \R^2)\ltimes\RR^3$.
We write the representations $\mathcal H_\sigma$ in the decomposition~\ref{eq:decompose_H}
as representations $\mathcal H_{t,r}$ of two parameters, which is realized by a unitarily equivalent model $\mathcal K_{t,r}$, see Section~4.3 of \cite{W} for a discussion.
We provide it here for the convenience of the reader.

Let $\rho_{t,r} : (SL(2, \R) \ltimes \R^2)\ltimes\RR^3 \to \mathcal U(\mathcal K_{t,r})$ be the irreducible, unitary representation with parameter $t$ and $r$ defined as follows.
Let $f\in \mathcal K_{t,r}$, $g=\begin{pmatrix}a & b & u_1\\
c & d & u_2\\
0 & 0 & 1\\
 \end{pmatrix}$ and $v = \left(\begin{array}{ll} v_1 \\ v_2 \\ v_3\end{array}\right).$ Then
$$
\rho_{t,r}(g) f(x, y,z) := e^{(p_1r+p_2t)\sqrt{-1}}f(D,E,FD),
$$
$$
\rho_{t,r}(v) f(x, y,z):= e^{(xv_1-yv_2-zv_3)\sqrt{-1}}f(x,y,z),
$$
where
\begin{align*}
D&= -cy+xa,\qquad E=yd-bx\notag\\
F&=(azd-adu_1x-cybu_2+au_2yd+cu_1bx-czb)D^{-1}.
\end{align*}
and
$$
\| f \|_{\mathcal K_{t,r}} = \| f \|_{L^3(\R^3)}.
$$
Let $Y_1=\begin{pmatrix}1 \\
0\\
0\\
 \end{pmatrix}$, $Y_2=\begin{pmatrix}0 \\
1\\
0\\
 \end{pmatrix}$ and $Y_3=\begin{pmatrix}0 \\
0\\
1\\
 \end{pmatrix}$.
A basis for $\sl(2, \R)\ltimes\RR^2$ is
\begin{gather*}
X=\begin{pmatrix}1 & 0 & 0\\
0 & -1& 0\\
0 & 0 & 0\\
 \end{pmatrix}\quad
 U_1=\begin{pmatrix}0 & 1 & 0\\
0 & 0& 0\\
0 & 0 & 0\\
 \end{pmatrix}\quad
 U_2=\begin{pmatrix}0 & 0& 1 \\
0 & 0& 0\\
0 & 0& 0\\
 \end{pmatrix}\\
 U_3=\begin{pmatrix}0 & 0 & 0 \\
0 & 0& 1\\
0 & 0& 0\\
 \end{pmatrix}
 \quad V_1=\begin{pmatrix}0 & 0 & 0 \\
1 & 0& 0\\
0 & 0& 0\\
 \end{pmatrix}.
 \end{gather*}
As in \cite{W},
 \begin{gather}\label{for:1}
 X=x\partial_x-y\partial_y,\qquad U_1=-x\partial_y,\qquad U_2=-x\partial_z,\notag\\
  U_3=y\partial_z+\sqrt{-1}tx^{-1},\notag\\
  V_1=-y\partial_x+\sqrt{-1}(r+tz)x^{-2},\notag\\
  Y_1=x\sqrt{-1},\qquad Y_2=-y\sqrt{-1},\qquad Y_3=-z \sqrt{-1}.
 \end{gather}
Taking the Fourier transformation with respect to $y$ and $z$, we get the \emph{dual  models} $\widehat{\mathcal K_{t,r}}=\{f(x,\xi_1,\xi_2):f\in L^2(\RR^3)\}$. Computing derived representations, we get
 \begin{gather}\label{for:4}
 X=x\partial_x+\xi_1\partial_{\xi_1}+I,\qquad U_1=-x\xi_1\sqrt{-1},\qquad U_2=-x\xi_2\sqrt{-1},\notag\\
  U_3=y\xi_2\sqrt{-1}+\sqrt{-1}tx^{-1},\notag\\
  V_1=-y\partial_x+\sqrt{-1}(r+t\partial_{\xi_2}\sqrt{-1})x^{-2},\notag\\
  Y_1=x\sqrt{-1},\qquad Y_2=\partial_{\xi_1},\qquad Y_3=\partial_{\xi_2}.
 \end{gather}
Consider the $\R^2$ action $\rho: \R^2 \to \textrm{Diffeo}(M)$ whose derivatives along the coordinate axes of $\R^2$ induce  the vector fields $U_1$ and $U_2$.
Let
\[
L_\rho := U_1^2 + U_2^2\,.
\]

\begin{prop}\label{prop:semidirect_NGH}
The operator $L_\rho$ is not globally hypoelliptic.
\end{prop}
\begin{proof}
Set $h(x,\xi_1,\xi_2)=f_1(x)f_2(\xi_2)f_3(\xi_2)$, where $f_i$, $1\leq i\leq3$ are Schwarz functions satisfying $f_1(x)=e^{-x^{-2}}$ if $\abs{x}\leq1$, $f_2(\xi_1)=1$, if $\abs{\xi_1}\leq1$ and $f_3(\xi_2)=1$, if $\abs{\xi_2}\leq1$. Also set $p=x^2\xi_2^2\cdot h$ Then $h$ and $p$ are smooth vectors in $\widehat{\mathcal K_{t,r}}$.

By relations in \eqref{for:4} for any $t,\,r\in\RR^2$ the equation $L_\alpha f=p$ has the form
\begin{align*}
  -x^2\xi_1^2f(x,\xi_1,\xi_2)-x^2\xi_2^2f(x,\xi_1,\xi_2)=p(x,\xi_1,\xi_2)
\end{align*}
in the dual model $\widehat{\mathcal K_{t,r}}$. Then
\begin{align}\label{for:3}
 f(x,\xi_1,\xi_2)=-\frac{p(x,\xi_1,\xi_2)}{x^2(\xi_1^2+\xi_2^2)}=-\frac{h(x,\xi_1,\xi_2)\cdot \xi_2^2}{\xi_1^2+\xi_2^2}.
\end{align}
It is clear that $f\in L^2(\RR^3)$. Hence $f\in\widehat{\mathcal K_{t,r}}$. Furthermore, we have
\begin{align*}
  Y_2f=-Y_2h\frac{\xi_2^2}{\xi_1^2+\xi_2^2}+2h\frac{\xi_2^2\xi_1}{(\xi_1^2+\xi_2^2)^2}.
\end{align*}
Since
\begin{align*}
  &\int_{\xi_1^2+\xi_2^2\leq1}\frac{\abs{h(x,\xi_1,\xi_2}^2)\cdot \xi_2^4\xi_1^2}{(\xi_1^2+\xi_2^2)^4}dxd\xi_1d\xi_2\\
  &=\int_{\xi_1^2+\xi_2^2\leq1}\frac{\abs{f_1(x)}^2\cdot \xi_2^4\xi_1^2}{(\xi_1^2+\xi_2^2)^4}dxd\xi_1d\xi_2\\
  &=\int\abs{f_1(x)}^2dx\cdot \int_{\xi_1^2+\xi_2^2\leq1}\frac{\xi_2^4\xi_1^2}{(\xi_1^2+\xi_2^2)^4}d\xi_1d\xi_2\\
  &>\infty,
\end{align*}
we see that $f$ is not a smooth vector in $\widehat{\mathcal K_{t,r}}$.
\end{proof}

\subsubsection{$SL(2, \R) \times SL(2, \R)$}

Let $\mathcal H$ be a representation space for $SL(2, \R)$.
The Lie algebra of $\sl(2, \R)$ has basis elements
\begin{equation}\label{equa:basis-sl(2, R)}
\begin{aligned}
U &= \left(\begin{array}{cc} 0 & 1 \\ 0 & 0 \end{array}\right)\,,
V &= \left(\begin{array}{cc} 0 & 0 \\ 1 & 0 \end{array}\right)\,,
X &= \left(\begin{array}{cc} 1/2 & 0 \\ 0 & -1/2 \end{array}\right)
\end{aligned}
\end{equation}
that respectively generate the upper and lower unipotent one-parameter
subgroups of $SL(2, \R)$ and the geodesic flow on $SL(2, \R)$.
Consequently, the Lie algebra of $SL(2, \R) \times SL(2, \R)$
is $\sl(2, \mathbb{R})\times \sl(2, \mathbb{R})$,
which has the basis
\begin{equation}\label{equa:basis-sl(2,R)^2}
\begin{array}{ccc}
&U_1 = U \times 0\,, & U_2 = 0 \times U\,, \\
&V_1 = V \times 0\,, & V_2 = 0 \times V\,, \\
&X_1 = X \times 0\,, & X_2 = 0 \times X\,.
\end{array}
\end{equation}

We take the operator $\triangle$ in \eqref{eq:laplace} to be
the Laplacian,
\[
\triangle := -X_1^2-1/2(U_1^2 + V_1^2) - X_2^2-1/2(U_2^2 +V_2^2)\,,
\]
and the measure space $(Z, \mu)$ is a product space, where $Z = \textrm{spec}(\Box_1) \times \textrm{spec}(\Box_2)$, and for $i = 1, 2$, $\Box_i$ is the Casimir operator for the $i_{\textrm{th}}$ component of $SL(2, \R)$,
\[
\Box_i := -X_i^2 - \frac{1}{2} (U_i V_i + V_i U_i)\,.
\]
In this way, the unitary representations $H_\sigma$ from \eqref{eq:laplace} are a direct sum of an at most countable collection of irreducible, unitary representations $H_\nu \otimes H_\theta$, where $(\nu, \theta) \in \textrm{spec}(\Box_1) \times \textrm{spect}(\Box_2)$.

As in the previous section, let $\rho$ be such that derivatives along the coordinate axes of $\R^2$ generate the vector fields $U_1$ and $U_2$.
Let
\[
L_\rho := U_1^2 + U_2^2\,.
\]

\begin{prop}\label{eq:sl2r2}
For any $n \geq 4$, $L_\rho$ is not globally hypoelliptic.
\end{prop}
\begin{proof}
This will follow from the next lemma and the
decomposition~\eqref{equa:SL2R)^2-unitary_decomp}.

\begin{lemma}\label{lemm:SL2R^2}
Let $\mathcal H_1$ and $\mathcal H_2$
be irreducible, unitary
representations of $SL(2, \R)$ in the principal series,
and let $\mathcal H = \mathcal H_1 \otimes \mathcal H_2$.
For any $\epsilon > 0$,
there exists a distribution
$\mathcal D \in \mathcal H^{-(3+\epsilon)} \setminus \mathcal H$
that satisfies
$$
\begin{aligned}
U_1 \mathcal D & = 0\,,\\
U_2 \mathcal D & = 0\,.
\end{aligned}
$$
on $\mathcal H^{-(4+\epsilon)}.$
\end{lemma}
\begin{proof}
Let $\epsilon > 0$.
For each $i = 1,2$, let $\{u_{k, i}\}_{k}\subset \mathcal H_{i}^\infty$
be the orthonormal basis for $\mathcal H_{i}$ that is constructed
using the ladder operators $X\pm i(V+U)$.
For each $i=1,2$, formulas (24) and (43) of \cite{FF1}
give a $U$-invariant distribution
$\mathcal D_i \in \mathcal H_i^{-(1/2+\epsilon)}$.
Define
\[
\mathcal D := \mathcal D_1 \otimes \mathcal D_2\,.
\]

\begin{lemma}\label{lemm:D-regularity, SL2R^2}
For any $\epsilon > 0$, we have
\[
\mathcal D \in \mathcal H^{-(3+\epsilon)},
\]
and $D \notin \mathcal H$.
\end{lemma}
\begin{proof}
The Hilbert space $\mathcal H$ has the orthonormal basis
$\{u_{j, 1} \otimes u_{k, 2}\}_{j, k} \subset \mathcal H^{\infty}$
for $\mathcal H$.
Pick an arbitrary element
$f = \sum_{j, k\in \mathbb Z} f_{j,k} u_{j, 1} \otimes u_{k,2} \in \mathcal H^{3+8\epsilon}$, where $(f_{j, k}) \subset \C$.
Then there is a constant $C_\epsilon > 0$ such that
\begin{align}\label{equa:D(f)-u_j_otimes_u_k}
|\mathcal D(f)| & \leq \sum_{j, k} |f_{j, k}|  |\mathcal D(u_{j, 1} \otimes u_{k, 2})| \notag \\
& = \sum_{j, k} |f_{j, k}|  |\mathcal D_1(u_{j, 1}) \mathcal D_2(u_{k, 2})|  \notag \\
& \leq C_\epsilon \sum_{j, k} |f_{j, k}|\|u_{j, 1}\|_{1/2+\epsilon} \|u_{k, 2}\|_{1/2+\epsilon}\,.
\end{align}

Formula (35) of \cite{FF1} gives that for any $i =1,2$
and any $\omega > 0$,
$$
 \Vert u_{j, 1}\Vert_{\omega}
\approx (1 + j^2)^{\omega/2} \Vert u_{j, 1}\Vert_0 \,.
$$
Then by the Holder inequality
and because the operators $\triangle_1$ and
$\triangle_2$ commute, we have
\begin{align}\label{equa:D(f)-u_j_otimes_u_k-2}
\eqref{equa:D(f)-u_j_otimes_u_k} & = C_\epsilon 
\left( \sum_{j, k \in \mathbb Z}|f_{j, k}|^2 \|(1+j^2)^{1/2+\epsilon}u_{j, 1}\|_{1/2+\epsilon}^2 \|(1+k^2)^{1/2+\epsilon}u_{k, 2}\|_{1/2+\epsilon}^2\right)^{1/2} \notag \\
& \leq C_\epsilon \left(\sum_{j, k} |f_{j, k}|^2 \|[(I+\triangle_1)(I+\triangle_2)]^{3/4+2\epsilon}u_{j, 1}\otimes u_{k,2}\|_{0}^2\right)^{1/2} \,.
\end{align}

Observe that for any $k \in \N$,
\begin{equation}\label{equa:triangle_i-triangle}
[(I+\triangle_1)(I+\triangle_2)]^k \leq (I+\triangle)^{2k}\,.
\end{equation}
Because $(I+\triangle_1)(I+\triangle_2)$
and $(I+\triangle)$ are essentially self-adjoint operators,
the operators $[(I+\triangle_1)(I+\triangle_2)]^k, (I+\triangle)^k$
are defined for all $k \in \R^+$ by the spectral theorem.
Then by interpolation, it follows that \eqref{equa:triangle_i-triangle}
holds for all $k \geq 0$.

Then
\begin{align}
\eqref{equa:D(f)-u_j_otimes_u_k-2} & \leq C_\epsilon \left(\sum_{j, k \in \mathbb Z} |f_{j, k}|^2 \Vert(I+\triangle)^{3/2+4\epsilon}u_{j, 1}\otimes u_{k,2}\Vert_{0}^2 \right)^{1/2} \notag \\
& \leq C_{\epsilon} \left(\sum_{j, k} |f_{j, k}|^2  \|u_{j, 1} \otimes u_{k, 2}\|_{3+8\epsilon}^2 \right)^{1/2} \notag \\
& = C_{s, \epsilon} \| f \|_{3+8\epsilon} \notag\,.
\end{align}

We now observe that $\mathcal D \notin \mathcal H$.
Specifically, because $\mathcal D_1 \notin \mathcal H_1$, there exists a sequence of
functions $(f_k)_{k \in \N} \subset \mathcal H_1$ such that for any $k$,
$\mathcal D_1(f_k) = 1$ and $\Vert f_k \Vert_0 \geq k$.
Then let $g \in \mathcal H_2$ be such that $\mathcal D_2(g) = 1$, and note that for any $k$, $f_k \otimes g \in \mathcal H$.
Then from definitions,
\[
\mathcal D(f_k \otimes g)  = \mathcal D_1(f_k) \mathcal D_2(g) = 1 \,.
\]
However, for any $k \in \N$,
\[
\Vert f_k \otimes g \Vert_0 = \Vert f_k \Vert_0 \Vert g \Vert_0 \geq k \Vert g \Vert_0\,.
\]
Hence, $\mathcal D \notin \mathcal H$.
\end{proof}

Finally, let $f \in \mathcal H^{4+\epsilon}$,
so $U_1 f \in \mathcal H^{3+\epsilon}$.
Then using the regularity of $\mathcal D$ Lemma~\ref{lemm:D-regularity, SL2R^2},
we get
\begin{align}
  U_1 \mathcal D(f)  & =   \mathcal D(U_1 f)  \notag \\
& =   \mathcal D\left( \sum_{j, k} f_{j, k} U_1 (u_{j, 1} \otimes u_{k, 2})\right)  \notag \\
& =   \sum_{j, k} f_{j, k} \mathcal D((U u_{j, 1}) \otimes u_{k, 2})  \,. \label{equa:U_1D-bound}
\end{align}
By assumption, $\mathcal D_1$ is a $U$-invariant distribution.
So
\begin{align}
\eqref{equa:U_1D-bound}
& =   \sum_{j, k} f_{j, k} \mathcal D_1(U u_{j, 1}) \mathcal D_2(u_{k, 2})  \notag \\
& = 0\,. \notag
\end{align}
Because $\mathcal H^{4+\epsilon}$ is dense in $\mathcal H^{3+\epsilon}$,
we conclude that $U_1 \mathcal D = 0$.

The statement $U_2 \mathcal D = 0$ is proved analogously.
Combining this with Lemma~\ref{lemm:D-regularity, SL2R^2}
gives Lemma~\ref{lemm:SL2R^2}.
\end{proof}

This completes the proof of Proposition~\ref{eq:sl2r2}.
\end{proof}

\subsection{${\rho}$ has a partially hyperbolic and a quasi-unipotent generator}

Let $n \geq 4$ and $d +2 \leq n$.
Let $E_{l j} \in M_{n}(\R)$ be the matrix with a one in the $(l, j)$
entry and a zero everywhere else.
Then the sets
\[
\mathcal B_d := \{ E_{l l} - E_{j j} : l < j \}\,, \ \ \ \ \mathcal B_u := \left\{ E_{l j} : l \neq j\right\}
\]
are bases for the diagonal and unipotent vector fields in $\sl(n, \R)$, respectively.
Let $X \in \mathcal B_d$ and $U \in \mathcal B_u$ be commuting elements
of $\sl(n, \R)$.
Suppose that for any $t, s \in \R$, the $\R^2$ action $\rho$ on $SL(n, \R)$ is given by
\begin{equation}\label{eq:rho-embed}
\rho(t, s) x:= \exp(t X) \exp(s U)x
\end{equation}
It embeds in $SL(d, \R) \times SL(2, \R)$,
with $\exp(X) \in SL(d, \R)$ and $\exp(U) \in SL(2, \R)$.

\begin{thm}\label{thm:main_mixed}
For any irreducible, unitary representation $\mathcal H$
of $SL(d, \R) \times SL(2, \R)$,
there is a finite regularity distribution $D \notin \mathcal H$
such that
$$
\begin{aligned}
& \frac{\partial}{\partial t} \rho(t, s) D = 0\,, \\
& \frac{\partial}{\partial s}\rho(t, s) D = 0 \,.
\end{aligned}
$$
\end{thm}

We begin by discussing the Sobolev unitary representation space of $SL(d, \R) \times SL(2, \R)$.
In this case the measure space $(Z, \mu)$ from the decomposition~\ref{eq:decompose_H} is a product space
and $\mathcal H_{\sigma}$ is a direct sum of an at most countable collection of irreducible, unitary representations of the form $H_\nu \otimes H_\theta$, where $H_\nu$, $H_\theta$ is a unitary representation of $SL(d, \R)$, $SL(2, \R)$, respectively.

Let $K$ be a maximal compact subgroup of $SL(d, \R)$. Denote by $\mathcal{K}$ its Lie algebra. Take an orthonormal basis $\{Y_j\}$ of $\mathcal{K}$, and set
\[
\triangle_1 := I-\sum Y_j^2\,.
\]
Then $\triangle_1$ belongs to the center of the universal enveloping algebra
of $\mathcal{K}$ and acts on smooth vectors of any representation space of $SL(d, \R)$.
Let
\[
\triangle_2 := I - X^2 - V^2 - U^2\,,
\]
where $X$, $V$ and $U$ respectively generate the geodesic and horocycle vector fields in $\sl(2, \R)$ and are given by formula~\ref{equa:basis-sl(2, R)}.

Now define the essentially self-adjoint operator $\triangle$ on $\mathcal H$ by
\[
\triangle := \triangle_1 + \triangle_2\,.
\]
\smallskip

Theorem~1.1 of \cite{FF1}
shows that there is a finite regularity, $U$-invariant distribution
in each irreducible, unitary representation of $SL(2, \R)$.
The next proposition shows that this is also true for the $X$ derivative
in irreducible, unitary representations of $SL(d, \R)$.

\begin{prop}\label{prop:XD=0}
For any non-trivial irreducible, unitary representation $(\pi, \mathcal H)$ of $SL(d, \R)$,
the following holds.
There is distribution $D \in \mathcal H^{-\alpha_1}$, where $\alpha_1\in \mathbb{N}$ depends only on $d$ when $d > 3$ and only on the spectral gap of $\pi$ when $d = 2$, such that
\[
X D = 0\,.
\]
\end{prop}

\begin{proof}
Since $X$ is partially hyperbolic, there is a vector $\nu\in \mathfrak{sl}(d, \R)$ in the root space of $SL(d, \R)$ such that
\begin{align}\label{for:2}
 [X,\nu]=\lambda \nu,\qquad \lambda<0.
\end{align}
It is harmless to assume that $\lambda<-1$. Let $S$ denote the connected subgroup with Lie algebra generated by $X$ and $\nu$ and let $S_1$ be the root subgroup of $\nu$. By the Howe-Moore Ergodicity theorem, the restricted representation $\pi|_S$ on $S$ contains non-nontrivial $S_1$-fixed vectors. Next, we will use Mackey theory \cite{M} to compute irreducible representations of $S$ without non-nontrivial $S_1$-fixed vectors.

\begin{thm}\label{th:1}(Mackey theorem, see \cite[Ex 7.3.4]{Zimmer}, \cite[III.4.7]{Margulis}) Let $S$ be a locally compact second countable group and $\mathcal{N}$ be an abelian closed
normal subgroup of $S$. We define the natural action of $S$ on the group of characters $\widehat{\mathcal{N}}$ of the group $\mathcal{N}$ by setting
\begin{align*}
    (s\chi)(\mathfrak{n}):=\chi(s^{-1}\mathfrak{n}s),\qquad s\in S,\,\chi\in \widehat{\mathcal{N}}, \,\mathfrak{n}\in \mathcal{N}.
\end{align*}
Assume that every orbit $S\cdot \chi$, $\chi\in \widehat{\mathcal{N}}$ is locally closed in $\widehat{\mathcal{N}}$. Then for
any irreducible unitary representation $\pi$ of $S$, there is a point $\chi_0\in \widehat{\mathcal{N}}$ with $S_{\chi_0}$
its stabilizer in $S$, a measure $\mu$ on $\widehat{\mathcal{N}}$ and an irreducible unitary representation  $\sigma$ of $S_{\chi_0}$ such that
\begin{enumerate}
  \item $\pi=\text{Ind}_{S_{\chi_0}}^S(\sigma)$,
  \item $\sigma\mid_{\mathcal{N}}=(\dim)\chi_0$,
  \item $\pi(x)=\int_{\widehat{\mathcal{N}}}\chi(x)d\mu(\chi)$, for any $x\in \mathcal{N}$; and $\mu$ is ergodically supported on the orbit $S\cdot \chi_0$.
\end{enumerate}
\end{thm}

Let $\mathbb E$ be the Hilbert space of functions on $\R^+$ with norm
\[
 \|f\|_{\mathbb{E}}=\|f\|_{L^2(\R^+, r^{-1}dr)}\,.
\]

\begin{lemma}\label{le:1}
The irreducible representations of $S$ without non-trivial $S_1$-fixed vectors are all equivalent to the following representation:
  \begin{gather*}
\beta: S\rightarrow \mathcal{B}(\mathbb{E})\\
\beta(\exp(\log s\cdot X),0)f(r)=f(s^{-1}r),\quad \forall\,s>0\\
\beta(e,\exp(t\nu))f(r)=tr^{-\Delta}\sqrt{-1}f(r); \quad \forall\,t\in\RR\,,
 \end{gather*}
where $\Delta$ is given by \eqref{for:2}. Computing derived representations, we get
\begin{align*}
 X=-r\partial_r,\qquad \nu=-r^{-\Delta}\sqrt{-1}.
\end{align*}
\end{lemma}
\begin{proof}
We note that $S_1$ is a normal subgroup of $S$. The group action is defined by
\begin{align*}
  &\exp(\log s\cdot X)\big( \exp(\log u\cdot X),\, \exp(t\nu)\big)\exp(-\log s\cdot X)\\
  &=\big( \exp(\log u\cdot X),\, \exp(s^\Delta t\nu)\big),
\end{align*}
for any $s,\,u>0$ and $t\in\RR$. Hence the dual action of the group $\{\exp(\log s\cdot X)\}_{s>0}$ on $S_1$ ($S_1$ is isomorphic to $\RR$) has two orbits: the original and the complement of the origin. The first factors to a representation of the group $\{\exp(\log s\cdot X)\}_{s>0}$. This means that $S_1$ acts trivially. Then we proceed to the second, which gives us the result by Mackey theorem.
\end{proof}
Let $S=\RR\ltimes\RR^2$, with Lie algebra $\{X,u_1,\,u_2\}$ satisfying $[X,u_i]=\lambda_i$, $i=1,\,2$. Let $S_i$ denote the abelian subgroup with Lie algebra $u_i$, $i=1,\,2$.
\begin{lemma}
The irreducible representations of $S$ without non-trivial $S_1$ or $S_2$-fixed vectors are induced representations and parameterized by
$s_0\in \RR\backslash0$ and the group action is defined by:
  \begin{gather*}
\beta^\delta_{s_0}: S\rightarrow \mathcal{B}(\mathbb{E}^\delta_{s_0})\\
\beta^\delta_{s_0}(\exp(\log s\cdot X),0)f(r)=f(s^{-1}r) \\
\beta^\delta_{s_0}(e,\exp(t_1u_1+t_2u_2))f(r)=e^{\small\text{$\sqrt{-1} ((-1)^\delta t_1r^{-\lambda_1}+s_0t_2r^{-\lambda_2})$}}f(r),
\end{gather*}
for any $t_1,\,t_2\in\RR$ and $s>0$, where $\delta\in \{+,\,-\}$; and
\begin{align*}
 \|f\|_{\mathbb{E}^\delta_{s_0}}=\|f\|_{L^2(\R^+, \frac{1}{r}dr)}.
\end{align*}
Computing derived representations, we get
\begin{align*}
 X=-r\partial_r,\qquad u_1=(-1)^\delta r^{-\lambda_1}\sqrt{-1},\qquad u_2=s_0 r^{-\lambda_2}\sqrt{-1}.
\end{align*}
\end{lemma}

Hence for $\pi|_S$, we have the direct integral decomposition: $\pi|_S=\int_Z \beta_zd\mu(z)$ and $\mathcal{H}=\int_Z \mathcal{H}_zd\mu(z)$ over some measure space $(Z,\mu)$ of irreducible unitary representations, i.e., $(\beta_z, \mathcal{H}_z)=(S,\mathbb{E})$.

For any function $f$ and any constant $c \in \C$, the Mellin transform is defined by
\[
\mathcal{M}(f,c)=\frac{1}{\sqrt{2\pi}}\int_0^\infty f(r) r^{c}dr\,.
\]
For $f\in \mathbb{E}$, we have
\begin{equation}\label{eq:mellin_unitary}
\|f\|^2_{\mathbb{E}} = \int_\RR |\mathcal{M}(f,0+s\sqrt{-1})|^2ds\,,
\end{equation}
and
\begin{align*}
\frac{d}{ds} \mathcal{M}(h,0+s\sqrt{-1}) = \sqrt{-1}\mathcal{M}(h(r)\cdot \log r,0+s\sqrt{-1})
\end{align*}
if $h(r)\cdot \log r\in \mathbb{E}$. Then we have
\begin{lemma}\label{le:3}
Suppose $f\in \mathbb{E}^2$ if the equation $Xg=\nu f$
has a solution $g\in \mathbb{E}^2$,
then
\[
\mathcal{M}(\nu f,0)=0\,.
\]
\end{lemma}

\begin{proof}
For any $f\in \mathbb{E}^2$, by \eqref{eq:mellin_unitary} and the Sobolev embedding theorem (with respect to the variable $s\in \R$),
we have
\begin{align*}
  |\mathcal{M}(\nu f,0+s\sqrt{-1})|_{C^0(\R)}< C(\|\nu f\|_{\mathbb{E}}+\|\nu f\cdot \log(r)\|_{\mathbb{E}}).
\end{align*}
Note that by assumption $\lambda<-1$, so $-1-2\lambda > 0$.
Then for any $0<\beta<-1-2\lambda$,
\begin{align}
 \|\nu f\cdot \log(r)\|_{\mathbb{E}}^2&\leq c_\beta\int_0^1 |\nu f(r)|^2 r^{-1-\beta}dr+C\int_1^\infty |\nu^2 f(r)|^2 r^{-1}dr \notag \\
 &=c_\beta\int_0^1 |f(r)|^2 r^{-\beta-1-2\lambda}dr+C\int_1^\infty |\nu^2 f(r)|^2 r^{-1}dr.  \label{eq:Mellin:sob:1}
\end{align}
Because $-\beta - 1- 2\lambda > 0$, we get
\[
 |\mathcal{M}(\nu f,0+s\sqrt{-1})|_{C^0(\R)}< C(\|f\|_{\mathbb{E}}+\|\nu^2 f\|_{\mathbb{E}})
\]
Finally, a computation shows that for any $Xh\in \mathbb{E}$,
\[
\mathcal{M}(Xh,0+s\sqrt{-1})=s\sqrt{-1}\mathcal{M}(h,0+s\sqrt{-1})\,.
\]
Now by assumption, there is $g\in \mathbb{E}^2$ such that $Xg=\nu f$.
So
\[
\begin{aligned}
\mathcal{M}(\nu f, 0+s\sqrt{-1}) = s\sqrt{-1}\mathcal{M}(g,0+s\sqrt{-1})\,.
\end{aligned}
\]
Setting $s = 0$ completes the proof of the lemma.
\end{proof}

We use $\mathcal{W}_{2,S}\subset \mathcal{H}$ to denote the $2$-order Sobolev space with the norm respect to $S$. Then
\begin{lemma}\label{le:2}
$\mathcal{H}^\infty$ is dense in $\mathcal{W}_{2,S}$.
\end{lemma}
\begin{proof}
For any $v\in \mathcal{H}$ and $\varphi\in C^\infty_c(S)$, let
\begin{align*}
\pi(\varphi)(v)=\int_S\varphi(s)\pi(s)vds,
\end{align*}
where $dg$ is a left invariant Haar measure.

Let $\mathcal{U}$ be the closure of $\mathcal{H}^\infty$ in $\mathcal{W}_{2,S}$.  For any $\omega\in \mathcal{W}_{2,S}$, choose $v_n\in \mathcal{H}^\infty$ such that $v_n\to \omega$ in $\mathcal{H}$. Then
it is clear that $\pi(\varphi)(v_n)\to \pi(\varphi)(\omega)$ in $\mathcal{W}_{2,S}$ since $\varphi$ is $C_{c}^\infty$ on $S$. We note that $\pi(\varphi)(v_n)\in \mathcal{H}^\infty$ for all $n$. Letting $\varphi$ run through a $C^\infty_c(S)$
approximate identity, we get $\omega\in \mathcal{U}$.
\end{proof}
The following result follows directly from Theorem 2.4.3 of \cite{KM}.
\begin{thm}\label{th:2}
For any unitary representation $\pi$ of $SL(d, \R)$, there is a universal constant $B>0$, a positive integer $l$ (dependent only on $SL(d, \R)$ and a positive constant $a$ (dependent only on the spectral gap $\sigma$ and $X$), such that
\begin{align*}
|\langle\pi(tX)v_1,v_2\rangle|\leq Be^{-|t|a}\|\triangle_1^lv_1\|\cdot \|\triangle_1^lv_2\|
\end{align*}
for any smooth vectors $v_1$, $v_2$ in the representation space of $\pi$.
\end{thm}
Now we are ready to prove Proposition~\ref{prop:XD=0}.
For any $v\in \mathcal H^\infty$, define
 \begin{align*}
   \mathcal{D}_v(\omega)=\int_{\RR} \langle\pi(tX)v,\omega\rangle dt,\qquad \omega\in \mathcal H^\infty.
 \end{align*}
The above theorem implies that there is a constant $\alpha_1 \in \N$ such that $\mathcal{D}_v$ is a distribution in $\mathcal H^{-\alpha_1}$ with estimates
\begin{align}
 |\mathcal{D}_v(\omega)|\leq C_{v , d, X, \sigma} \|\triangle_1^{\alpha_1}\omega\|\,,
\end{align}
for some constant $C_{v, d, X, \sigma} > 0$.  By the fundamental theorem of calculus and the Howe-Moore Ergodicity theorem, we get $XD_v=0$.
Hence we proved the first part of Proposition~\ref{prop:XD=0}.

Next, we will show that there is $v\in \mathcal H^\infty$ such that $\mathcal{D}_v \neq 0$.
Assume to the contrary that $\mathcal{D}_v=0$ for any $v\in \mathcal H^\infty$. We recall the following result:
\begin{lemma}(Corollary~4.3 of \cite{W1})
Suppose $(\pi,\mathcal{H})$ is a unitary representation of $SL(d,\R)$ such that
$\pi$ has a spectral gap and $v\in \mathcal{H}^\infty$. If $\mathcal{D}_v=0$, then the equation $Xg=v$ has a solution $g\in \mathcal{H}^\infty$.
\end{lemma}
Hence the assumption that $\mathcal{D}_v=0$ for any $v\in \mathcal H^\infty$ implies that $Xg=v$ has a solution $g\in \mathcal{H}^\infty$ once $v\in \mathcal H^\infty$.

Recall that $(Z, \mu)$ is a measure space (see Lemma~\ref{le:1}).  Fix a finite measurable set $\mathcal{E}\subset Z$.  For each $z\in \mathcal{E}$, let $f_z=h$, where
  $h\in C_c^\infty(\RR)$ with support inside the interval $[1,2]$ satisfying $\mathcal{M}(\nu h,0)=1$. It is clear that $h\in \mathbb{E}^\infty$. Let
  \[
  \bar{h}=\int_Z \chi_\mathcal{E}(z)f_zd\mu(z)\,,
  \]
 where $\chi_\mathcal{E}$ is the characteristic function of $\mathcal{E}$. Then $\bar{h}\in \mathcal{W}_{2,S}$, and by Lemma~\ref{le:2}, there is $v_n\in \mathcal{H}^\infty$ such that $v_n\to \bar{h}$ in $\mathcal{W}_{2,S}$.

 By assumption, for every $n$, the equation $Xg_n=\nu v_n$
has a solution $g_n\in \mathcal H^\infty$.
 We have the decomposition of $v_n=\int_Z v_n(z)d\mu(z)$ for all $n$.
 Hence, for every $n$ and almost every $z\in \mathcal{E}$,
  we have $v_n(z)\to h$ in $\mathbb{E}^2$, and the equation $Xg_n(z)=\nu v_n(z)$ has a solution $g_n(z)\in \mathbb{E}^\infty$.  It follows from Lemma~\ref{le:3} that
$\mathcal{M}(\nu v_n(z),0)=0$.  Then using \eqref{for:1} and the fact that $\mathcal{M}$ is linear in its first variable, we have
  \begin{align*}
 1&=|\mathcal{M}(\nu v_n-\nu h,0)|\leq |\mathcal{M}(\nu v_n(z)-\nu h,0+s\sqrt{-1})|_{C^0(\R)}\\
 &< C(\|v_n(z)-h\|_{\mathbb{E}}+\|\nu^2 (v_n(z)-h )\|_{\mathbb{E}})\,.
\end{align*}
Noting that the right side goes to $0$ as $n\to \infty$, we have a contradiction.  This concludes the proof of Proposition~\ref{prop:XD=0}.
\end{proof}

\begin{prop}\label{prop:SL2R^2}
Let $\mathcal H = \mathcal H_1 \otimes \mathcal H_2$,
where $\mathcal H_1$ and $\mathcal H_2$ are irreducible, unitary representations
of $SL(d, \R)$ and $SL(2, \R)$, respectively.
There is a number $\alpha_2 := \alpha_2(\mathcal H_2) > 1/2$
and a distribution $D$ such that
$D \in \mathcal H^{-(\alpha_1+\alpha_2)} \setminus \mathcal H$
and
$$
\begin{aligned}
(X\times 0) D = (0\times U) D = 0\,.
\end{aligned}
$$
\end{prop}
\begin{proof}
Let $\{u_{k}\}_{k}\subset \mathcal H_{2}^\infty$
be an orthonormal basis for the unitary representation space $SL(2,\R)$.
Let $D_1$ be as in Proposition~\ref{prop:XD=0}.  By Theorem~1.1 of \cite{FF1},
there is a number $\alpha_2 := \alpha_2(\mathcal H_2)$
and a $U$-invariant distribution $D_2$ such that
$D_2 \in \mathcal H_2^{-\alpha_2}$.  Define
\[
D := D_1 \otimes D_2\,.
\]

\begin{lemma}\label{lemm:D-regularity, mixed}
Let $\alpha_1$ be as in Proposition~\ref{prop:XD=0}.
Then
\[
D \in \mathcal H^{-(\alpha_1+\alpha_2+1)} \setminus \mathcal H\,.
\]
\end{lemma}
\begin{proof}
Let $u_k$ be an orthogonal basis of $\mathcal{H}_2$ and let $K$ be the maximal compact subgroup in $SL(2,\R)$. Then for any $\omega\in \mathcal{H}_1\otimes \mathcal{H}_2$ we can write $\omega=\sum_{n}\omega_n \otimes u_n$, where $\omega_n\in \mathcal{H}_1$.
So $\{\omega_n \otimes u_n\}_n$ is orthogonal in $\mathcal H$.
If $\omega\in \mathcal{H}^{\alpha_1+\alpha_2}$, then
\begin{align*}
 \sum_{n}\|\triangle_1^{\alpha_1}\omega_n\|^2
 \cdot \|n\cdot\triangle_2^{\alpha_2}u_n\|^2&= \|(\triangle^{\alpha_1}_1\otimes (K\triangle_2^{\alpha_2}))(\omega)\|^2\\
 & \leq \|\omega\|^2_{\alpha_1+\alpha_2+1}\,.
\end{align*}

Hence,
\begin{align*}
 |\mathcal{D}_v\otimes D_2(\omega)| &= |\sum_{n} \mathcal{D}_v(\omega_n) D_2(u_n)| \leq \sum_{n} C_v\|\triangle_1^{\alpha_1}\omega_n\|
 \cdot \|\triangle_2^{\alpha_2}u_n\| \\
 &=\sum_{n} C_v\|\triangle_1^{\alpha_1}\omega_n\|
 \cdot \|n\cdot\triangle_2^{\alpha_2}u_n\|\cdot |n|^{-1}\\
 &\leq C_v \big(\sum_{n}\|\triangle_1^{\alpha_1}\omega_n\|^2
 \cdot \|n\cdot\triangle_2^{\alpha_2}u_n\|^2\big)^\frac{1}{2}\cdot (\sum_{n} |n|^{-2})^\frac{1}{2}\\
 &\leq C_v\|\omega\|_{\alpha_1+\alpha_2+1}\,.
\end{align*}
\end{proof}

We can now finish the proof of Proposition~\ref{prop:SL2R^2}.
From the above lemma, it suffices to prove that $D$
vanishes on $(X \times 0)$ and $(0 \times U)$.
Let $f \in \mathcal H^\infty$.
Because $D_1$ is an $X$-invariant distribution,
we get
\begin{align}
 \vert(X\times 0) D(f)\vert
& = \vert D\left((X\times 0) \sum_{j, k} f_{j, k} u_{j, 1} \otimes u_{k, 2}\right)\vert \notag \\
& \leq \sum_{j, k} |f_{j, k}| |D\left((X u_{j, 1}) \otimes u_{k, 2}\right)| \notag \\
& = \sum_{j, k} |f_{j, k}| |D_1(X u_{j, 1}) D_2(u_{k, 2})| \notag \\
& = \sum_{j, k} |f_{j, k} |(X D_1) (u_{j, 1}) D_2(u_{k, 2})| \notag \\
& = 0\,. \notag
\end{align}
So $(X\times 0) D = 0$.

The statement for $(0 \times U)$ is proved analogously.
This completes the proof of Proposition~\ref{prop:SL2R^2}.
\end{proof}

\begin{proof}[Proof of Theorem~\ref{thm:main_mixed}]
This follows immediately from Proposition~\ref{prop:SL2R^2}.
\end{proof}




\section{Proof of Theorem \ref{main2}}\label{thm_main2}

As we already mentioned and discussed in Section \ref{discreteD}, without loss of generality we may assume that $D$ is discrete subgroup in $R$. In this section we assume $R$ is solvable, in which case by a theorem of Mostow \cite[Theorem E. 3]{S}, $R/D$ is of finite volume if and only if $R/D$ is compact.

If the action $\rho$ contains a quasiunipotent ergodic generator, then the group $R$ is of a special kind. Recall the following
\begin{definition}
A solvable group $R$ is a class (I) group if for every $g\in R$ the spectrum of $Ad(g)$ is contained in the unit circle.
\end{definition}

Since $\rho$ contains a quasi-unipotent ergodic flow we may now use the following result from \cite{FFRH}:

\begin{lemma}[Lemma 4.10 \cite{FFRH}]
If there is an ergodic and quasi-unipotent homogeneous flow on $R/D$, then $R$ is of class(I).
\end{lemma}

This in particular implies that all the elements of the action $\rho$ are  quasi-unipotent.

\begin{thm} [Theorem 4.9 in \cite{FFRH}, \cite{S}, \cite{A}]Any ergodic homogeneous flow on a class (I) compact solvable manifold $R/D$ is smoothly conjugate to a homogeneous flow on a compact nilmanifold.
\end{thm}

The nilmanifold in the theorem above  does not depend on the ergodic flow itself, so the conjugacy from the above Theorem conjugates the whole action $\rho$ to a homogeneous $\mathbb R^k$ action on the nilmanifold. (The action is homogeneous from the fact that any flow commuting with the Diophantine flow on the nilmanifold is  also homogeneous.)



\medskip
\medskip

\end{document}